\documentclass[a4paper]{amsart}

\usepackage[english]{babel}
\usepackage[utf8]{inputenc}
\usepackage{amssymb,amsmath,amsthm,amsfonts}
\usepackage{graphicx}
\usepackage[usenames,dvipsnames]{xcolor}
\usepackage{cancel}
\usepackage[normalem]{ulem}
\usepackage{dsfont}
\usepackage{enumitem}
\usepackage{microtype}
\usepackage[showonlyrefs]{mathtools}
\mathtoolsset{showonlyrefs=true}

\usepackage{tikz}
\usetikzlibrary{matrix, calc, positioning, decorations.pathreplacing, arrows.meta}
\usepackage[colorlinks]{hyperref}

\newcommand{\N}{\mathbb{N}}

\newcommand{\R}{\mathbb{R}}

\newcommand{\Om}{\Omega}

\renewcommand{\H}{H^1_0}
\renewcommand{\phi}{\varphi}

\def\N{\mathbb{N}}

\newtheorem{proposition}{Proposition}[section]
\newtheorem{theorem}[proposition]{Theorem}

\newtheorem{lemma}[proposition]{Lemma}
\theoremstyle{definition}

\newtheorem{remark}[proposition]{Remark}

\title[Rigidity of the ball for Hartree-type energies]{Rigidity of the ball for Hartree-type energies and second shape derivative via a continuity argument}

\author{Dario Mazzoleni}
\address{Dipartimento di Matematica F. Casorati\\
		Universit\`a di Pavia\\
		Via Ferrata 5, 27100 Pavia, Italy}
\email{dario.mazzoleni@unipv.it}

\author{Riccardo Moraschi}
\address{Dipartimento di Matematica F. Casorati\\
		Universit\`a di Pavia\\
		Via Ferrata 5, 27100 Pavia, Italy}
\email{r.moraschi1@campus.unimib.it}

\author{Berardo Ruffini}
\address{Dipartimento di Matematica\\
Universit\`a di Bologna\\
Piazza di Porta S.Donato 5, 40126, Bologna, Italy}
\email{berardo.ruffini@unibo.it}

\begin{document}

\begin{abstract}
We study the second shape derivative of an optimal design functional involving a nonlocal term in order to prove rigidity of the ball. Our approach relies on a continuity argument that relates the second variation to that of the corresponding local problem. The model serves as a prototype for a broader class of shape optimization problems to which the same strategy may be adapted.
\end{abstract}

\maketitle

\section{Introduction}\label{sec:intro}

We study the persistence of spherical optimal shapes under a weak nonlocal
perturbation of a spectral energy. For an open set
$\Omega\subset\mathbb R^3$ of finite measure and $q\geq0$, let
\begin{equation}\label{eq:EqOmegau}
 E_q(u,\Omega)
 :=\int_\Omega |\nabla u|^2\,dx
   +\frac q2\int_\Omega\int_\Omega
       \frac{u(x)u(y)}{|x-y|}\,dx\,dy,
\end{equation}
and define
\begin{equation}\label{eq:EqOmega}
 E_q(\Omega)
 :=\min\left\{E_q(u,\Omega):
       u\in H^1_0(\Omega),\ \int_\Omega u^2\,dx=1\right\}.
\end{equation}
Our aim is to characterize the minimizers of
\begin{equation}\label{eq:infEqOmega}
 \min\left\{E_q(\Omega):
       \Omega\subset\mathbb R^3\text{ open},\ |\Omega|=|B_1|\right\}
\end{equation}
for sufficiently small $q$, where $B_1$ is the unit ball. The parameter $q$
measures the strength of the Coulomb interaction relative to the Dirichlet
energy. The scaling identity $E_q(r\Omega)=r^{-2}E_{qr^4}(\Omega)$ also
identifies this weak-coupling regime with a small-volume regime at fixed
coupling.

At $q=0$, the functional reduces to the first Dirichlet eigenvalue
$\lambda_1(\Omega)$, and the Faber--Krahn inequality identifies balls as
minimizers under a volume constraint. It can be strenghtened to its sharp quantitative form, that was first proved 
by Brasco, De Philippis and Velichkov~\cite{BrascoDePhilippisVelichkov}.
For $q>0$, however, the two terms compete: for nonnegative functions,
symmetric decreasing rearrangement decreases the Dirichlet energy but
increases the Coulomb interaction~\cite{LiebLoss}. Moreover, no sign
constraint is imposed in~\eqref{eq:EqOmega}; replacing $u$ by $|u|$ does not
decrease the energy in general. Thus neither the optimal shape nor positivity of the
optimal state follows directly from rearrangement.

This competition is related to the Ohta--Kawasaki model for diblock
copolymers~\cite{OK86} and its small-volume-fraction
analysis~\cite{ChoksiPeletier}. For perimeter energies with Coulomb
repulsion, Kn\"upfer and Muratov~\cite{KM2} established spherical
minimality at small mass, including the three-dimensional Coulomb case.
In the spectral setting, Mazzoleni and Ruffini~\cite{MazzoleniRuffini}
proved small-repulsion rigidity for the sum of the first Dirichlet
eigenvalue and a Riesz interaction of the domain. Here the interaction
depends on the optimizing state, which changes with both $q$ and $\Omega$;
the energy is therefore not a sum of two independently prescribed shape
functionals. Another related problem is the charge-qubit model of
Mazzoleni, Muratov and Ruffini~\cite{MazzoleniMuratovRuffini}, whose
Hartree term involves $u^2(x)u^2(y)$, whereas the interaction
in~\eqref{eq:EqOmegau} is quadratic in $u$.

For the present functional, the authors, in a previous paper~\cite{MazzoleniMoraschiRuffini} proved existence and convergence
of optimal shapes to the ball by combining surgery, variational
convergence and one-phase free boundary regularity. 
We recall here the result.
\begin{theorem}[\cite{MazzoleniMoraschiRuffini}]\label{thm:main}
There exist $q^*>0$ and $\alpha\in(0,1)$ such that, for every
$0<q\leq q^*$, problem~\eqref{eq:infEqOmega} admits a minimizer.
Every minimizer $\Omega_q$ satisfies (up to translations and sets of zero Lebesgue measure)
\[
 \partial\Omega_q
 =\{(1+\varphi_q(x))x:x\in\partial B_1\},
 \qquad
 \|\varphi_q\|_{C^{2,\alpha}(\partial B_1)}\longrightarrow0
 \quad\text{as }q\to 0^+.
\]
\end{theorem}
This reduces the rigidity question to a neighborhood of the ball, but
does not exclude non-spherical minimizers converging to it. The purpose
of this paper is to establish the rigidity conjectured
in~\cite{MazzoleniMoraschiRuffini}.
\begin{theorem}[Rigidity]\label{thm:intro-rigidity}
There exists $q_0>0$ such that, for every $q\in[0,q_0]$, the unit ball
minimizes~\eqref{eq:infEqOmega}, uniquely up to translations (and
modifications on sets of zero capacity). 
\end{theorem}

\textbf{Idea of the proof.} Since we already know that the optimal domains for \eqref{eq:infEqOmega} are $C^{2,\alpha}$-nearly spherical, we can perform a second order shape derivative to prove that the ball is the only minimizer.
Namely, let us formally consider a function $f_q(t)$. We aim to prove the following inequality 
\[f_q(1)\geq f_q(0)+c,\]
where $c$ is just a positive quantity. In the setting of shape functionals, this is roughly equivalent to proving that the function $f_q$ admits a minimum in $0$.
So we consider the Taylor-Lagrange expansion around $0$ and we evaluate it at $t=1$:
\begin{equation}\label{eq:}
f_q(1)=f_q(0)+f_q'(0)+\frac{1}{2}f_q''(0)+\int_0^1(1-s)(f_q''(s)-f_q''(0))\,ds.
\end{equation}
After the computations of the first and second derivative, it can be shown that $f'(0)=0$. So, in order to conclude, some estimates on the remainder terms are needed . This is usually the most technical part when addressing a shape derivative argument. To avoid it, the idea is to employ that $f_q$ is, for $q$ small, very close to $f_0$, which in our context is a well-known functional. So the idea is to inherit all the needed estimate from $f_0$ by a continuity argument. The delicate point is to understand how close is $f''_q(t)$ to $f''_0(t)$: indeed it is quite easy to show that as $f''_q(t) \to f''_0(t)$ as $q\to 0$ but no quantitative estimate about how small is the gap $f''_q(t)-f''_0(t)$ is obtained. It turns out that it is enough to show the convergence of the operator associated to $f''_q$ and $f''_0$. A deeper explanation is postponed to Remark \ref{rem:spiegazione convergenza operatoriale}.  

A priori, this convergence argument should work by an hierarchy of the functional spaces. Indeed since  \[\int_\Omega\int_\Omega
       \frac{u(x)u(y)}{|x-y|}\,dx\,dy=4\pi[u]^2_{\dot H^{-1}(\R^3)}\leq C(|\Omega|)[u]^2_{\H(\Omega)}, \]
 this means that we are adding a lower order seminorm to $[u]^2_{\H(\Omega)}$. Then as $q$ goes to $0$, it is reasonable that not only the energy but also the derivatives of it converge to the ones of the stronger norm.
\vspace{2pt}

The proof transfers the second-order stability of the first Dirichlet eigenvalue to $E_q$. The relevant framework is the theory of stability through second variations developed by Dambrine and Lamboley~\cite{DambrineLamboley}: the Taylor remainder must be controlled in the same, weaker norm in which the shape Hessian is coercive. In our case this is the $H^{1/2}$ norm of the boundary displacement, although the perturbations are small in $C^{2,\alpha}$. This distinction is essential for obtaining estimates uniform over all sufficiently small deformations.

\textbf{Section 2.} We first pull the state equation back to the fixed ball and apply the implicit function theorem at the normalized first Dirichlet eigenpair. Simplicity of the eigenvalue makes the linearized equation invertible once the normalization is included. For small deformations, derivatives of the pulled-back Coulomb kernel retain the integrable singularity $|x-y|^{-1}$, allowing smooth dependence on the shape and on $q$. Elliptic estimates then yield uniform convergence of the states and their boundary derivatives as $q\to 0^+$. On the ball, bijectivity of the implicit function and rotational invariance give radial symmetry, hence stationarity under the volume constraint.

\textbf{Section 3.} The main estimate concerns the linearized states associated with arbitrary normal boundary data $h\in H^{1/2}$. In particular Lemma \ref{le:a_q coerciva e stime} gives the linear and continuous dependence of the linearized states with respect to the boundary value. Substitution into the second-variation formula yields, for the normal graphs
$\partial\Omega_t^\varphi=\{(1+t\varphi(x))x:x\in\partial B_1\}$,
\[
 \sup_{\varphi\in\mathcal A_\delta}\sup_{t\in[0,1]}
 \bigl\|\mathcal L_{q,t}^\varphi-\mathcal L_{0,t}^\varphi\bigr\|_{
 \mathcal L(H^{1/2}(\partial\Omega_t^\varphi),
             H^{-1/2}(\partial\Omega_t^\varphi))}
 \longrightarrow0.
\]
Here $\mathcal A_\delta$ consists of volume-constrained normal graphs with $C^{2,\alpha}$ norm at most $\delta$, and the operator $\mathcal L_{q,t}^\varphi$ represents the second variation along these paths. Crucially, the resulting error is $o_q(1)\|\varphi\|_{H^{1/2}(\partial B_1)}^2$. Note that the pointwise convergence of the second derivatives at each fixed deformation would only give only an error of the type $o_q(1)$, insufficient to conclude.

\textbf{Section 4.} To handle the volume constraint and translation invariance along paths that need not preserve either volume or barycentre, we follow the penalization approach of~\cite{DambrineLamboley} and set
\[
 J_q(\Omega)=E_q(\Omega)+\mu_q|\Omega|
       +K\bigl(|\Omega|-|B_1|\bigr)^2
       +K|\operatorname{Bar}(\Omega)|^2,
 \qquad
 \mu_q=|\partial_\nu u_{q,B_1}|^2_{\big|_{\partial B_1}}.
\]
Here $\operatorname{Bar}$ denotes the barycentre and $K$ is sufficiently large. The multiplier makes the ball stationary, while the quadratic terms control the volume and translation modes. The coercivity and Taylor-remainder estimates for $J_0$, together with the operator-norm comparison, imply
\[
 E_q(\Omega^\varphi)-E_q(B_1)
 \geq c\|\varphi\|_{H^{1/2}(\partial B_1)}^2
\]
for small $q$ and small $\varphi$, provided $|\Omega^\varphi|=|B_1|$ and $\operatorname{Bar}(\Omega^\varphi)=0$. Theorem~\ref{thm:main} places every minimizer in this neighborhood after translation, so this local estimate proves global rigidity.

\vspace{0.5cm}
{\bf Note.}
This work started in the Spring of 2026 and was mostly carried out in June 2026 when D. Mazzoleni and R. Moraschi visited B. Ruffini at the Department of Mathematics of the University of Bologna.
When the paper was in the last phases of preparation,  the work~\cite{Topaloglu} appeared on arXiv. In~\cite{Topaloglu} the author proves a result analogous to Theorem~\ref{thm:intro-rigidity}, but with different techniques (in particular without a shape derivative argument). 
Moreover, in the references of~\cite{Topaloglu}, we discovered the anonymous paper~\cite{Anonymous}. This latter paper, proves rigidity for functional proposed in~\cite{MazzoleniMuratovRuffini} and similar to~\eqref{eq:EqOmegau}, where the nonlocal energy is $D(u^2,u^2)$ instead of $D(u,u)$. Although the techniques therein are more in the spirit of the present paper (and allow also to treat critical points), they still do not employ a second shape derivative which is a main feature of our work.

\section{Differentiability of the functional}
We consider deformations of the form $\Phi_\theta(y) = y + \theta(y)$, mapping $B_1$ onto $\Om_\theta = \Phi_\theta(B_1)$, with $\theta \in C^{2,\alpha}(\overline{B_1}, \R^3)$. In this section we show that the optimal function $u_q$ on $\Phi_\theta(B_1)$ is differentiable as $\theta$ goes to $0$.

Defining $J_\theta = \det(I + D\theta)$, $A_\theta = (I + D\theta)^{-T}$ and the rescaled state function $v = u \circ \Phi_\theta$, the operator $F(\theta, q, v, \lambda)$ is given by:
\[
    F: C^{2,\alpha}(\overline{B_1}, \R^3) \times \R \times \left( H^2(B_1) \cap H_0^1(B_1) \right) \times \R \longrightarrow L^2(B_1) \times \R
\]

\begin{equation}\label{eq: funzionale del dini}
        F(\theta, q, v, \lambda) = 
    \begin{pmatrix}
        -\text{div}\left(J_\theta A_\theta^T A_\theta \nabla v\right) - \lambda J_\theta v + \frac{q}{2} J_\theta V_\theta[v] \\[8pt]
        \int_{B_1} v^2 J_\theta \, dy - 1
    \end{pmatrix}
\end{equation}
where \[V_\theta[v](y) = \int_{B_1} \frac{v(z)}{|\Phi_\theta(y) - \Phi_\theta(z)|} J_\theta(z) \, dz.\]
 For the sake of simplicity, we also denote the operator
\begin{equation}\label{eq:Lq}
L_{q,\Omega}(u)=-\Delta u+\frac{q}{2}\int_\Omega\frac{u(x)}{|x-y|}\,dx,\qquad u\in H^1_0(\Omega), \;\int_\Omega u^2=1,
\end{equation}
and its Dirichlet eigenvalues $\lambda_{k,q}(\Omega)$, for $k\in \N$.

Setting $\lambda_0$ the first eigenvalue of the Dirichlet Laplacian on $B_1$ and $v_0$ the corresponding unique strictly positive, $L^2$-normalized eigenfunction, we shall apply the implicit function theorem to show that,
near $(0,0,v_0,\lambda_0)$, the zero set of $F$ is a smooth
manifold.

\begin{theorem}\label{thm: dini}
Let $\lambda_0$ be the first eigenvalue of the Dirichlet Laplacian on $B_1$ and $v_0$ be the corresponding unique strictly positive, $L^2$-normalized eigenfunction. 
There exist neighborhoods $U \subset C^{2,\alpha}(\overline{B_1}, \R^3) \times \R$ of $(0,0)$, $V\subseteq \left( H^2(B_1) \cap H_0^1(B_1) \right) \times \R$ of $(v_0,\lambda_0)$ and a unique $C^\infty$ map 
\[
    g: U \longrightarrow V, \quad g(\theta,q) = (v_{\theta,q}, \lambda_{\theta,q})
\]
such that $g(0,0) = (v_0, \lambda_0)$ and $F(\theta, q, v_{\theta,q}, \lambda_{\theta,q}) = (0,0)^T$ for all $(\theta,q) \in U$. In particular considering $ (\theta, q, v, \lambda)\in U\times V$, then
\begin{equation}\label{eq:uniqueness dini}
  F(\theta, q, v, \lambda)=(0,0)^T  \iff g(\theta,q)= (v,\lambda).
\end{equation}
 Moreover, up to possibly shrinking $U$ and $V$, $(v_{\theta,q}, \lambda_{\theta,q})$ are, respectively, the first eigenfunction and the first eigenvalue for the operator $L_{q,\Omega}$ defined in~\eqref{eq:Lq}, with Dirichlet boundary conditions.
\end{theorem}

\begin{proof}
The proof relies on verifying the hypotheses of the Dini's Theorem in Banach spaces (see \cite[Theorem 1.4]{AmbrosettiMalchiodi}).  Since $J_0 = 1$ and $A_0 = I$, then $F(0, 0, v_0, \lambda_0) = (0,0)^T$, recalling that for $\theta=0,q=0$ we can use the equation satisfied by $u$.
Along a perturbation direction $(w, \mu)$, the directional derivative is defined as:
\[
    D_{(v,\lambda)}F(0,0,v_0,\lambda_0)[w, \mu] = \left. \frac{d}{ds} F(0,0, v_0+sw, \lambda_0+s\mu) \right|_{s=0}.
\]
So
\begin{align*}
    F(0,0, v_0+sw, \lambda_0+s\mu) &= 
    \begin{pmatrix}
        -\text{div}(\nabla(v_0+sw)) - (\lambda_0+s\mu)(v_0+sw) \\[8pt]
        \int_{B_1} (v_0+sw)^2 \, dy - 1
    \end{pmatrix} \\[10pt]
    &= 
    \begin{pmatrix}
        -\Delta v_0 - s\Delta w - \lambda_0 v_0 - s\mu v_0 - s\lambda_0 w - s^2\mu w \\[8pt]
        \int_{B_1} (v_0^2 + 2swv_0 + s^2w^2) \, dy - 1
    \end{pmatrix}.
\end{align*}
Taking the derivative with respect to $s$ and evaluating a $s=0$ we obtain:
\[
    D_{(v,\lambda)}F(0,0,v_0,\lambda_0)[w, \mu] = 
    \begin{pmatrix}
        -\Delta w - \mu v_0 - \lambda_0 w \\[8pt]
        2 \int_{B_1} w v_0 \, dy
    \end{pmatrix}.
\]
To show that this linear map is an isomorphism, we fix a datum $(g, c) \in L^2(B_1) \times \R$ and seek a unique pair $(w, \mu) \in \left(H^2(B_1) \cap H_0^1(B_1)\right) \times \R$ solving:
\[
    \begin{cases}
        -\Delta w - \lambda_0 w - \mu v_0 = g \\
        2 \int_{B_1} w v_0 \, dy = c.
    \end{cases}
\]
Since $\lambda_0$ is a simple eigenvalue, by the Fredholm alternative (e.g., \cite[Lemma 5.7.3 ]{henrotpierre2018}), there exists a unique weak solution $w \in H_0^1(B_1)$ to the resulting orthogonal system. Because $g + \mu v_0 \in L^2(B_1)$, standard elliptic regularity theory (see \cite[Theorem 8.12]{GilbargTrudinger}) ensures $w \in H^2(B_1)$, proving bijectivity. 

Furthermore, $ D_{(v,\lambda)}F(0,0,v_0,\lambda_0)$ is continuous since it is linear and bounded. As a linear, continuous, and bijective operator between Banach spaces $\left(H^2(B_1) \cap H_0^1(B_1)\right) \times \R$ and $L^2(B_1) \times \R$, the Open Mapping Theorem guarantees that its inverse is also continuous. Thus, $D_{(v,\lambda)}F(0,0,v_0,\lambda_0)$ is a topological isomorphism.

Finally, we must justify the $C^\infty$ regularity of the operator, paying particular attention to the nonlocal term. We factorize the spatial deformation as:
\[
    \Phi_\theta(x) - \Phi_\theta(y) = (x-y) + \left( \int_0^1 D\theta(tx + (1-t)y) \, dt \right)(x-y)= (I + M_\theta)(x-y),
\]
where $M_\theta(x,y) := \int_0^1 D\theta(tx + (1-t)y) \, dt$. Setting $\xi = \frac{x-y}{|x-y|}$, the square of the distance becomes:
\[
    |\Phi_\theta(x) - \Phi_\theta(y)|^2 = |x-y|^2 \left( \xi^T G_\theta \xi \right), \quad \text{with } G_\theta = (I + M_\theta)^T (I + M_\theta).
\]
Consequently, the singular kernel takes the form:
\[
    \frac{1}{|\Phi_\theta(x) - \Phi_\theta(y)|} = \frac{1}{|x-y|} \frac{1}{\sqrt{\xi^T G_\theta \xi}}.
\]
For a sufficiently small neighborhood of $\theta=0$ in $C^{2,\alpha}$, the norm $\|M_\theta\|$ is strictly bounded. The quadratic form satisfies:
\[
    \xi^T G_\theta \xi = 1 + 2\xi^T M_\theta \xi + |M_\theta \xi|^2 \ge 1 - 2\|M_\theta\| > 0.
\]
Because this denominator is uniformly bounded away from zero, the function $\theta \mapsto 1/\sqrt{\xi^T G_\theta \xi}$ is of class $C^\infty$. The regularity of the other terms is the same as the one done in \cite[Lemma 4.8]{DambrineLamboley}.

Finally, we need to show that, possibly in smaller neighborhoods $U\times V$, $(v_{\theta,q}, \lambda_{\theta,q})$ are the first eigenfunction and the first eigenvalue for the operator $L_{q,\Omega}$  with Dirichlet boundary conditions.
First, it is possible to check that, for all $k\geq 1$,\[
0\leq \lambda_{k,q}(\Omega_\theta)-\lambda_{k,0}(\Omega_\theta)\leq C q,
\]
since the nonlocal operator is uniformly bounded.
Then, after possibly shrinking $U$, we can assume that \[
\lambda_{2,q}(\Omega_\theta)-\lambda_{1,q}(\Omega_\theta)\geq \frac{\lambda_{2,0}(\Omega_\theta)-\lambda_{1,0}(\Omega_\theta)}{2}\geq c>0.
\]
Thus, up to possibly shrinking $V$, we can assume that the branch provided by the Dini's Theorem stays below a threshold, while every second eigenvalue stays above it. 
Therefore $\lambda_{\theta,q}=E_q(\Omega_\theta)$ is the first eigenvalue and its eigenspace is one-dimensional. 
\end{proof}

\begin{remark}\label{rem:grad g limitato}
    Being $g$ as in Theorem \ref{thm: dini} and up to consider a smaller neighborhood $U$, we can suppose $||\partial_qg||_{L^\infty}\leq M$ where $M>0$ is a positive universal constant.
\end{remark}
\begin{theorem}\label{thm:simmetric solution on ball}
There exists an universal $q^*$, such that for any $0<q\leq q^*$ the optimal function $u_q$ for the functional $E_q(B_1)$ is unique (up to change sign). In particular it is radially symmetric.
\end{theorem}

\begin{proof}
Let $u_q$ be an optimal function for the variational problem $E_q(B_1)$. Then, $u_q$ satisfies the associated Euler--Lagrange equation:
\[
    \begin{cases}
        -\Delta u_q = \lambda_q u_q - \frac{q}{2} \left( u_q \ast \frac{1}{|\cdot|} \right)  & \text{in } B_1, \\[4pt]
        u_q = 0 & \text{on } \partial B_1,
    \end{cases}
\]
where $\lambda_q=E_q(B_1)$.

Let $u_0$ be the unique, strictly positive, $L^2$-normalized first eigenfunction of the Dirichlet Laplacian on $B_1$. Subtracting the equation satisfied by $u_0$ from the one for $u_q$, and setting $w = u_q - u_0$, we obtain the PDE for the difference:
\[
    \begin{cases}
        -\Delta w = (\lambda_q - \lambda_0) u_0 + \lambda_q w - \frac{q}{2} \left( u_q \ast \frac{1}{|\cdot|} \right) & \text{in } B_1, \\[4pt]
        w = 0 & \text{on } \partial B_1.
    \end{cases}
\]

Since the function $$v\in L^2(B_1)\to \int_{B_1} \int_{B_1} \frac{v(x)v(y)}{|x-y|}\, dxdy\in \R$$ is continuous (see \cite[Proposition 2.2]{MazzoleniMoraschiRuffini}), then by $\Gamma$-convergence we know that $E_q(\cdot, B_1) \xrightarrow{\Gamma} E_0(\cdot, B_1)$ as $q \to 0^+$ in weak $H^1_0(B_1)$ topology. Then up to consider $-u_q$, it holds that $u_q \to u_0$ in $L^2(B_1)$ and $\lambda_q \to \lambda_0$.

By standard elliptic regularity theory \cite[Theorem 8.13]{GilbargTrudinger}, we obtain that $w\to 0$ in $H^2(B_1)$, so $u_q \to u_0$ strongly in $H^2(B_1)$.

Recalling the operator $F(\theta, q, v, \lambda)$ defined in \eqref{eq: funzionale del dini}, we have that $F(0, q, u_q, \lambda_q) = (0,0)^T$ for all $q>0$ small enough. Then by Theorem \ref{thm: dini}, we know there exists a neighborhood $V$ of $(u_0, \lambda_0)$ in $\left( H^2(B_1) \cap H_0^1(B_1) \right) \times \R$ such that, for $q$ small enough, the problem admits a \emph{unique} solution branch $(\tilde{u}_q, \tilde{\lambda}_q) \in V$.
Since $(u_q, \lambda_q) \to (u_0, \lambda_0)$ in $\left( H^2(B_1) \cap H_0^1(B_1) \right) \times \R$, for $q$ sufficiently small, then $(u_q, \lambda_q)\in V$ . 

Finally, since the energy on the ball $B_1$ is rotationally invariant, the uniqueness of the solution implies that the solution $u_q$ is radially symmetric. 
\end{proof}

Let $\Omega_\theta = \Phi_\theta(B_1)$ be the perturbed domain associated with the deformation $\Phi_\theta = \mathrm{Id} + \theta$. Consider the map $g: (\theta, q) \mapsto (v_{q,\theta}, \lambda_{q,\theta})$ mapping the deformation parameter and the charge to the rescaled state on $B_1$ and the eigenvalue. By Theorem~\ref{thm: dini}, $g$ is smooth on a neighborhood $U$ of $(0,0)$. 

\begin{theorem}\label{thm:convergence}
Let $\mathcal{A} \subset C^{2,\alpha}(\overline{B_1}, \mathbb{R}^3)$ be a closed and bounded neighborhood of 0 and $q^* > 0$ such that $\mathcal{A} \times [0, q^*] \subset U$. Let $u_{q,\theta}$ and $u_{0,\theta}$ be the state functions defined on $\Omega_\theta$. As $q \to 0^+$, the following limits hold:
\begin{align*}
    u_{q,\theta} &\to u_{0,\theta} \quad \text{in } C^{2,\alpha}(\Omega_\theta), \\
    \lambda_{q,\theta} &\to \lambda_{0,\theta} \quad \text{in } \mathbb{R},
\end{align*}
uniformly with respect to $\theta \in \mathcal{A}$. 
\end{theorem}

\begin{proof}
Since $\mathcal{A} \times [0, q^*]\subset U$, in light of Remark \ref{rem:grad g limitato}, the function $g$ is Lipschitz continuous with constant $M$. Thus for all $\theta \in \mathcal{A}$ and $q \in [0, q^*]$:
\begin{equation*}
    \|v_{q,\theta} - v_{0,\theta}\|_{L^2(B_1)} \le Mq, \quad 
    |\lambda_{q,\theta} - \lambda_{0,\theta}| \le Mq.
\end{equation*} 

We transfer these $L^2$ convergences to the domain $\Omega_\theta$. 
Denoting $J_\theta = \det(\nabla \Phi_\theta)$, we obtain:
\begin{equation*}
   \|u_{q,\theta} - u_{0,\theta}\|_{L^2(\Omega_\theta)}^2 = \int_{B_1} (v_{q,\theta} - v_{0,\theta})^2 J_\theta \, dy \le \|J_\theta\|_{L^\infty(B_1)} \|v_{q,\theta} - v_{0,\theta}\|_{L^2(B_1)}^2.
\end{equation*}
Since $\mathcal{A}$ is a bounded subset of $C^{2,\alpha}$, the Jacobian $J_\theta$ is uniformly bounded from above for all $\theta \in \mathcal{A}$. Therefore, $\|u_{q,\theta} - u_{0,\theta}\|_{L^2(\Omega_\theta)} \to 0$ uniformly as $q \to 0^+$. 

To establish the higher-order convergence, consider the PDE satisfied by the difference $w_q := u_{q,\theta} - u_{0,\theta}$ on $\Omega_\theta$:
\begin{equation*}
    \begin{cases}
        -\Delta w_q = \lambda_{q,\theta} w_q + (\lambda_{q,\theta} - \lambda_{0,\theta}) u_{0,\theta} - \frac{q}{2} v_{u_{q,\theta}} & \text{in } \Omega_\theta, \\
        w_q = 0 & \text{on } \partial\Omega_\theta.
    \end{cases}
\end{equation*}
Applying standard elliptic regularity (\cite[Theorem 6.6, Theorem 8.15, Theorem 8.29]{GilbargTrudinger}) on the $C^{2,\alpha}$ domain $\Omega_\theta$, the $C^{2,\alpha}(\Omega_\theta)$ norm of $w_q$ is bounded by:
\begin{equation*}
    \|w_q\|_{C^{2,\alpha}(\Omega_\theta)} \le C(\theta) \Big( \|w_q\|_{L^2(\Omega_\theta)}+ 
    |\lambda_{q,\theta} - \lambda_{0,\theta}| \|u_{0,\theta}\|_{L^\infty(\Omega_\theta)} + q \|v_{u_{q,\theta}}\|_{C^{0,\alpha}(\Omega_\theta)} \Big).
\end{equation*}
Because the domains $\Omega_\theta$ are parameterized by the set of deformations $\mathcal{A}$, the ellipticity constants $C(\theta)$ are uniformly bounded. From the previous step, $|\lambda_{q,\theta} - \lambda_{0,\theta}|$ vanish uniformly as $q \to 0^+$. 
Since $\Delta v_{u_q,\theta}=-4\pi u_{q,\theta}$, the regularity estimate of $v_{u_q,\theta}$ are inherited from the one of $u_{q,\theta}$.
\end{proof}
\begin{remark}\label{rem:positivity}
We choose the sign of \(u_{q,\theta}\) to be nonnegative. Indeed for sufficiently small deformations of the ball, the Hopf boundary estimate holds uniformly:
\[-\partial_{\nu_\theta}u_{0,\theta}\ge c>0
\qquad\text{on }\partial\Omega_\theta,\]
with \(\nu_\theta\) the outward unit normal. Uniform boundary regularity gives a fixed tubular neighborhood od $\partial\Omega_\theta$ in which \(u_{0,\theta}\) increases uniformly along inward normals, while \(u_{0,\theta}\) is uniformly bounded away from zero outside this neighborhood. The uniform \(C^1\) convergence \(u_{q,\theta}\to u_{0,\theta}\) therefore preserves positivity in both regions. Consequently, after decreasing \(q^*>0\), independently of \(\theta\) the nonnegativity of the optimal function holds for every \(q\in[0,q^*]\).
\end{remark}

\section{Convergence of the operator}

Let $\alpha\in(0,1)$ and, for $\delta>0$ small, set
\[ \mathcal A_\delta := \left\{\varphi\in C^{2,\alpha}(\partial B_1):\|\varphi\|_{C^{2,\alpha}(\partial B_1)}\leq\delta, \quad |\Omega^\varphi|=|B_1| \right\}, \]
where
\[ \partial\Omega^\varphi = \left\{(1+\varphi(x))x:x\in\partial B_1\right\}.\]
For every $\varphi\in\mathcal A_\delta$, let us set
\begin{equation}\label{eq:radial deformation of B1}
    \Omega_t^\varphi:=\{x+xt\phi(x):x\in B_1\}, \qquad t\in[0,1]. 
\end{equation}
Observe that all the previous results, obtained with the general deformation $\Phi=Id+\theta$, hold for the choice $\theta(x)=t\phi(x)x$ where $x\in\partial B_1$ and $t\in[0,1]$.
\begin{remark}
  Since \(\partial\Omega_t\) is the normal graph of \(t\phi\) over \(\partial B_1\), as in \cite[Section~4]{DambrineLamboley}, and \(\|t\phi\|_{C^{2,\alpha}(\partial B_1)}\leq\delta\), the domains \(\Omega_t\) form a uniformly \(C^{2,\alpha}\) family for \(\delta\) sufficiently small. Consequently, the constants in the trace and extension theorems, as well as the geometric constants in elliptic regularity estimates, can be chosen uniformly in \(t\) and \(\phi\). Moreover, to well define $\Omega_t$ as a deformation with a vector field, we can extend $\phi\nu$ in a neighborhood of $\partial B_1$ as $\phi\big(\frac{x}{|x|}\big)\frac{x}{|x|}$ and then smoothly extend to $\R^3$. From now on, $\phi\nu$ is meant to be the extension on the whole space. 

\end{remark}

For $q\geq 0$, let $(u_{q,t}^\varphi,\lambda_{q,t}^\varphi)$ denote the positive, $L^2$-normalized solution of 
\[ 
\begin{cases}
-\Delta u_{q,t}^\varphi = \lambda_{q,t}^\varphi u_{q,t}^\varphi -\dfrac q2 v_{u_{q,t}^\varphi} & \text{in }\Omega_t^\varphi, \\[2mm]
u_{q,t}^\varphi=0& \text{on }\partial\Omega_t^\varphi, \\[1mm]
\displaystyle
\int_{\Omega_t^\varphi}(u_{q,t}^\varphi)^2\,dx=1,
\end{cases}
\]

Given $h\in H^{1/2}(\partial\Omega_t^\varphi)$, let $u_{q,t}^{\prime,\varphi}[h]\in H^1(\Omega_t^\varphi)$ be the unique solution of
\begin{equation}
\label{eq:linearized-state-h}
\begin{cases}
-\Delta u_{q,t}^{\prime,\varphi}[h] = \lambda_{q,t}^{\prime,\varphi}[h]u_{q,t}^\varphi +\lambda_{q,t}^\varphi u_{q,t}^{\prime,\varphi}[h] -\dfrac q2v_{u_{q,t}^{\prime,\varphi}[h]} &\text{in }\Omega_t^\varphi, \\[2mm]
u_{q,t}^{\prime,\varphi}[h] = -\partial_{\nu_t}u_{q,t}^\varphi\,h &\text{on }\partial\Omega_t^\varphi, \\[2mm]
\displaystyle
\int_{\Omega_t^\varphi} u_{q,t}^\varphi u_{q,t}^{\prime,\varphi}[h]\,dx=0, 
\end{cases}
\end{equation}
with
\[ \lambda_{q,t}^{\prime,\varphi}[h] = -\int_{\partial\Omega_t^\varphi} |\nabla u_{q,t}^\varphi|^2h\,d\mathcal H^2. \]


\begin{lemma}
There exist $\delta,q_0>0$ such that, for every $\varphi\in\mathcal A_\delta$, $q\in[0,q_0]$, $t\in[0,1]$, and $h\in H^{1/2}(\partial\Omega_t^\varphi)$, problem \eqref{eq:linearized-state-h} admits a unique solution $u_{q,t}^{\prime,\varphi} [h]\in H^1(\Omega_t^\varphi)$.
\end{lemma}

\begin{proof}
We follow the argument of \cite[Lemma~5.7.3]{henrotpierre2018}. Set
\[
\mathcal L_{q,t}^\varphi w
:=-\Delta w+\frac q2
\int_{\Omega_t^\varphi}\frac{w(y)}{|x-y|}\,dy.
\]
For $\delta$ and $q_0$ sufficiently small, the local uniqueness provided by  Theorem \ref{thm: dini} implies that $\lambda_{q,t}^\varphi$ is simple.

Given $(Z,\Lambda)\in H^{-1}(\Omega_t^\varphi)\times\mathbb R$, we first seek a pair $(w,\mu)\in H_0^1(\Omega_t^\varphi)\times\mathbb R$ satisfying
\[
\begin{cases}
(\mathcal L_{q,t}^\varphi-\lambda_{q,t}^\varphi)w
-\mu u_{q,t}^\varphi=Z,\\[2mm]
\displaystyle 2\int_{\Omega_t^\varphi}u_{q,t}^\varphi w=\Lambda.
\end{cases}
\]
Since $\mathcal L_{q,t}^\varphi$ is self-adjoint with compact resolvent and
\[
\ker(\mathcal L_{q,t}^\varphi-\lambda_{q,t}^\varphi)
=\operatorname{span}\{u_{q,t}^\varphi\},
\]
the Fredholm alternative shows that the first equation is solvable precisely when
\[
\mu=-\langle Z,u_{q,t}^\varphi\rangle.
\]
For this choice, its solutions are of the form $w=w_0+s u_{q,t}^\varphi$, and the second equation uniquely fixes
\[
s=\frac{\Lambda}{2}
-\int_{\Omega_t^\varphi}u_{q,t}^\varphi w_0.
\]
Thus the system admits a unique pair $(w,\mu)$.

Now choose a lifting $G_h\in H^1(\Omega_t^\varphi)$ of
$-\partial_{\nu_t}u_{q,t}^\varphi h$, and take
\[
Z=-(\mathcal L_{q,t}^\varphi-\lambda_{q,t}^\varphi)G_h,
\qquad
\Lambda=-2\int_{\Omega_t^\varphi}u_{q,t}^\varphi G_h.
\]
Let $(w,\mu)$ be the corresponding solution and define
\[
u_{q,t}^{\prime,\varphi}[h]:=G_h+w.
\]
Thus it holds
\[
\begin{aligned}
\mu
=-\langle Z,u_{q,t}^\varphi\rangle
=\int_{\partial\Omega_t^\varphi}
G_h\,\partial_{\nu_t}u_{q,t}^\varphi
=-\int_{\partial\Omega_t^\varphi}
|\nabla u_{q,t}^\varphi|^2h
=\lambda_{q,t}^{\prime,\varphi}[h].
\end{aligned}
\]
Hence $u_{q,t}^{\prime,\varphi}[h]$ solves \eqref{eq:linearized-state-h}. Uniqueness follows from that of the lifted system.
\end{proof}

The implicit function theorem provides smooth dependence in the topology of regular domain deformations. For the stability argument, however, we need estimates in the weaker boundary space $H^{1/2}$. We therefore study the linearized state for arbitrary $h\in H^{1/2}(\partial\Omega_t^\varphi)$, with constants uniform in both the domain and the $q$.
\begin{lemma}\label{le:a_q coerciva e stime}
There exist $\delta_1,q_1>0$ such that  that for all $\phi \in \mathcal{A}_{\delta_1}$, $t\in[0,1]$ and $q\in [0,q_1]$ it holds \begin{equation}\label{eq:stime u_q'}
        || u_{q,t}^{\prime,\varphi}[h]||_{H^1(\Omega_t^\varphi)}\leq C ||h||_{H^{1/2}(\partial \Omega_t^\varphi)},  
    \end{equation}
        \begin{equation}\label{eq:stime u'_q-u'_0}
       || u_{q,t}^{\prime,\varphi}[h]- u_{0,t}^{\prime,\varphi}[h]||_{H^1(\Omega_t^\varphi)}\leq o_q(1)||h||_{H^{1/2}(\partial \Omega_t^\varphi)}.
    \end{equation}
    for all $h\in H^{1/2}(\partial\Omega_t^\phi)$.
\end{lemma}

\begin{proof}
    Let us define \[V_q(\Omega_t^\varphi)=\{v:v\in H_0^1(\Omega_t^\varphi),  (v,u_q)_{L^2(\Omega_t^\varphi)}=0\}.\]
    Moreover we set 
    \[ \lambda_{2,q}(\Omega_t^\varphi)= \min_{v\in V_q(\Omega_t^\varphi)} \frac{E_q(v,\Omega_t^\varphi)}{||v||^2_{L^2(\Omega_t^\varphi)}}.\]
    Setting, for $v\in V_q(\Omega_t^\varphi)$, $a_q(v,v)=E_q(v,\Omega_t^\varphi)-\lambda_q||v||^2_{L^2(\Omega_t^\varphi)}$, then \[ a_q(v,v)\geq (\lambda_{2,q}(\Omega_t^\varphi)-\lambda_q(\Omega_t^\varphi))||v||^2_{L^2 (\Omega_t^\varphi)}\qquad \forall v\in V_q(\Omega_t^\varphi).\]
Moreover $\lambda_{2,q}(\Omega_t^\varphi)$ and $\lambda_q(\Omega_t^\varphi)$ converge uniformly, with respect to $t\in[0,1]$ and $\phi\in \mathcal{A}_{\delta_0}$, to $\lambda_2(\Omega_t^\varphi)$ and $\lambda_0(\Omega_t^\varphi)$. Indeed the convergence $\lambda_q(\Omega_t^\phi)$ follows from Theorem \ref{thm:convergence}. Meanwhile, since $(\Delta+qv_u)$ has compact and self-adjoint resolvent, so there exists a orthonormal base composed by eigenfunction of $L^2(\Omega_t^\phi)$. This implies that

\[  \lambda_{2,q}(\Omega_t^\phi)
=\inf_{\substack{S\subset H_0^1(\Omega_t^\phi)\\ \dim S=2}}
\sup_{\substack{v\in S\\ \|v\|_{L^2(\Omega_t^\phi)}=1}}E_q(v).\]
Moreover \[
0\le D_{\Omega^\phi_t}(v,v)\le2\int_{\Omega_t^\phi}\int_{\Omega_t^\phi} 
\frac{v(x)^2+v(y)^2}{|x-y|}\le C(|\Omega_t^\phi|)\|v\|_{L^2(\Omega_t^\phi)}^2
\]
Choosing $\delta\leq \frac{1}{2}$, we obtain that $|B_{1/2}|\leq|\Omega_t^\phi|\leq |B_{3/2}|$ and so applying the min--max principle gives
\[
0\le\lambda_{2,q}(\Omega_t^\phi)-\lambda_{2,0}(\Omega_t^\phi)
\le Cq.
\]

    Since $\phi\in \mathcal{A}_{\delta_0}$, then $\lambda_0(\Omega_t^\varphi)\leq \lambda_0(B_{1-\delta_0})$ and $\lambda_{2,0}(\Omega_t^\varphi)\geq \lambda_{2,0}(B_{1+\delta_0})$. So, up to choosing $\delta_0$ and $q_0$ small enough, 
\begin{equation}\label{eq:L^2 coercivity}
        a_q(v,v)\geq \lambda_{2,q}(\Omega_t^\varphi)-\lambda_q(\Omega_t^\varphi)||v||^2_{L^2(\Omega_t^\phi)}\geq \frac{\lambda_{2,0}(B_1)-\lambda_0(B_1)}{2}||v||^2_{L^2(\Omega_t^\phi)}=C||v||^2_{L^2(\Omega_t^\phi)},
\end{equation}
    Since $\lambda_q(\Omega_t^\phi)\leq C$ for all $\phi\in\mathcal{A}_{\delta_0}$ due to Theorem \ref{thm: dini}, then
\begin{equation}\label{eq:gradient L^2 coercivity}
         ||\nabla v||^2_{L^2(\Omega_t^\phi)}\leq a_q(v,v)+\lambda_q(\Omega_t^\phi)||v||^2_{L^2(\Omega_t^\phi)}\leq (1+C)a_q(v,v). 
\end{equation}
    Thus by \eqref{eq:L^2 coercivity} and \eqref{eq:gradient L^2 coercivity}, $a_q$ is coercive with universal costant $c^*$. 
    Let us now prove \eqref{eq:stime u_q'}.  Let $E:H^{1/2}(\partial\Omega_t^\varphi)\to H^1(\Omega_t^\varphi)$ be a lifting operator and let $y\in H^1_0(\Omega_t^\varphi)$ be such that
    \[ u_{q,t}^{\prime,\varphi}[h]= \widetilde{E}_q[h]+y.\]
    where $\widetilde{E}_q[h]=E(-\partial_{\nu}u_{q}\,h)-(E(-\partial_{\nu}u_{q}\,h),u_q)_{L^2(\Omega_t^\varphi)}u_q$.
    By construction $(u_{q,t}^{\varphi},y)_{L^2(\Omega_t^\varphi)}=0$, and so 
    \[0=a_q(u_{q,t}^{\prime,\varphi}[h],y)=a_q(u_{q,t}^{\prime,\varphi}[h]-y,y)+a_q(y,y).\]
    Then by coercivity, estimate on the lifting operator and continuity of the multiplication map $C^{1,\alpha}(\partial  \Omega_t^\varphi)\times H^{1/2}(\partial \Omega_t^\varphi)\to  H^{1/2}(\partial \Omega_t^\varphi)$ (see \cite[Lemma 5.3]{DiNezzaPalatucciValdinoci}) it holds
    \[ c^*||y||^2_{H^1(\Omega_t^\varphi)}\leq a_q(y,y)=-a(\widetilde{E}_q[h],y)\leq||\widetilde{E}_q[h]||_{H^1(\Omega_t^\varphi)}||y||_{H^1(\Omega_t^\varphi)}\leq  C||h||_{H^{1/2}(\partial\Omega_t^\varphi)}||y||_{H^1(\Omega_t^\varphi)}.\]
    Finally
    \[||u_{q,t}^{\prime,\varphi}[h]||_{H^1(\Omega_t^\varphi)}=||\widetilde{E}_q[h]+y||_{H^1(\Omega_t^\varphi)}\leq ||\widetilde{E}_q[h]||_{H^1(\Omega_t^\varphi)}+||y||_{H^1(\Omega_t^\varphi)} \leq C||h||_{H^{1/2}(\partial\Omega_t^\varphi)}.\]
    Same computations holds for \eqref{eq:stime u'_q-u'_0} by studying the operator $(-\Delta-\lambda_0)$ and letting $y\in H^1_0(\Omega_t^\varphi)$ such that
\begin{equation}\label{eq:definizione y in u'_q-u'_0}
        u_{q,t}^{\prime,\varphi}[h]-u_{0,t}^{\prime,\varphi}[h]= \widetilde{E}_0[h]+(u_{q,t}^{\prime,\varphi}[h],u_0)_{L^2(\Omega_t^\varphi)}u_0+y,
\end{equation}
    where in this case
    \[\widetilde{E}_q[h]=E(-(\partial_{\nu}u_{q}-\partial_{\nu}u_{0})\,h)-(E(-(\partial_{\nu}u_{q}-\partial_{\nu}u_{0})\,h),u_0)_{L^2(\Omega_t^\varphi)}u_0.\]
    
  Indeed letting $w:=u_{q,t}^{\prime,\varphi}[h]- u_{0,t}^{\prime,\varphi}[h]$, it satifies
\[    \begin{cases}
        (-\Delta-\lambda_0)w=f_{q,t}&\text{in}\,  \Omega^\phi_t\\
        w=-\big(\frac{\partial u_q}{\partial \nu_t}-\frac{\partial u_0}{\partial \nu_t}\big)h & \text{on} \, \partial\Omega_t^\phi\\
        \int_{\Omega_t^\phi}u_{q,t}^{\prime,\varphi}[h]u_q=0 \\
   \int_{\Omega_t^\phi}u_{0,t}^{\prime,\varphi}[h]u_0=0 
    \end{cases}
    \]
    where \[ f_{q,t}= \lambda_{q,t}^{\prime,\varphi}[h](u_{q,t}^\varphi-u_{0,t}^\varphi)+( \lambda_{q,t}^{\prime,\varphi}[h]- \lambda_{0,t}^{\prime,\varphi}[h])u_{0,t}^\varphi+(\lambda_q(\Omega^\phi_t)-\lambda_0(\Omega^\phi_t))u_{q,t}^{\prime,\varphi}[h]-\frac{q}{2}v_{u_{q,t}^{\prime,\varphi}[h]}.\]
    Moreover by \eqref{eq:stime u_q'} and Theorem \ref{thm:convergence}, it holds
    \[ ||f_{q,t}||_{L^2(\Omega_t^\phi)}\leq o_q(1)||h||_{H^{1/2}(\partial \Omega_t^\phi)}. \]
    Recalling the definition of $y$ in \eqref{eq:definizione y in u'_q-u'_0}, $y\in V_0$ and so it holds
    \begin{align}
        C||y||^2_{H^1(\Omega_t^\phi)}&\leq a_0(y,y)=-a(w-y,y)+\langle f_{q,t},y\rangle \\&\leq||\widetilde{E}_q[h]||_{H^1(\Omega_t^\varphi)}||y||_{H^1(\Omega_t^\varphi)}+||f_{q,t}||_{L^2(\Omega_t^\phi)}||y||_{H^1(\Omega_t^\varphi)}\\&\leq  o_q(1)||h||_{H^{1/2}(\partial\Omega_t^\varphi)}||y||_{H^1(\Omega_t^\varphi)}.
    \end{align}
    where we use that $||\widetilde{E}_q[h]||_{H^1(\Omega_t^\varphi)}\leq o_q(1)||h||_{H^{1/2}(\Omega_t^\phi)}$ since $\partial_{\nu}u_{q}-\partial_{\nu}u_{0}$ converges to $0$ in $C^{1,\alpha}(\Omega_t)$. Since $(u_{q,t}^{\prime,\varphi}[h],u_q)=0$, by \eqref{eq:stime u_q'} and Theorem \ref{thm:convergence} it holds
    \begin{align}
        |(u_{q,t}^{\prime,\varphi}[h],u_0)u_0|=|(u_{q,t}^{\prime,\varphi}[h],u_q-u_0)u_0|\leq C \|h\|_{H^{1/2}(\Omega_t^\varphi)}\|u_q-u_0\|_{L^2(\Omega_t^\varphi)}\leq o_q(1)\|h\|_{H^{1/2}(\Omega_t^\varphi)}.
    \end{align}
    and then we conclude as before.
    \end{proof}

Define the symmetric bilinear form $\mathcal Q_{q,t}^\varphi$ on $H^{1/2}(\partial\Omega_t^\varphi)$ by 
\begin{align}\label{eq:bilinear-second-variation}
\mathcal Q_{q,t}^\varphi(h,k):=& \,2\int_{\Omega_t^\varphi} \nabla u_{q,t}^{\prime,\varphi}[h] \cdot \nabla u_{q,t}^{\prime,\varphi}[k]\,dx-2\lambda_{q,t}^\varphi \int_{\Omega_t^\varphi} u_{q,t}^{\prime,\varphi}[h] u_{q,t}^{\prime,\varphi}[k]\,dx\\
&+q\int_{\Omega_t^\varphi} v_{u_{q,t}^{\prime,\varphi}[h]} u_{q,t}^{\prime,\varphi}[k]\,dx+ 2\int_{\partial\Omega_t^\varphi} \left((H_t-\beta_t)|\nabla u_{q,t}^\varphi|^2+\frac q2 v_{u_{q,t}^\varphi} |\nabla u_{q,t}^\varphi| -\nabla^2u_{q,t}^{\varphi}[\nabla u_{q,t}^{\varphi}]\theta_t \right)hk\,d\mathcal H^2
 \end{align}
where \[\theta_t=\frac{|x|}{x\cdot\nu_t}\Bigg(\frac{x}{|x|}-\bigg(\frac{x}{|x|}\cdot\nu_t\bigg)\nu_t\Bigg) \qquad \beta_t:=\frac{1}{x\cdot\nu_t}.\] 
Let
\[
\mathcal L_{q,t}^\varphi \in \mathcal L \left(H^{1/2}(\partial\Omega_t^\varphi), H^{-1/2}(\partial\Omega_t^\varphi) \right) \]
be the operator associated with this bilinear form, namely
\[
\left\langle \mathcal L_{q,t}^\varphi h,k \right\rangle_ {H^{1/2}(\partial\Omega_t^\varphi),H^{-1/2}(\partial\Omega_t^\varphi)} = \mathcal Q_{q,t}^\varphi(h,k).
\]

\begin{lemma}\label{lem:convergence operator}
For all $q\in [0,q_1]$ it holds
\begin{equation}
\label{eq:uniform-operator-convergence}
\sup_{\varphi\in\mathcal A_{\delta_1}} \sup_{t\in[0,1]} \left\| \mathcal L_{q,t}^\varphi-\mathcal L_{0,t}^\varphi\right\|_{\mathcal L\left(H^{1/2}(\partial\Omega_t^\varphi),H^{-1/2}(\partial\Omega_t^\varphi)\right)}\leq o_q(1).
\end{equation}
\end{lemma}
\begin{proof}
For notational simplicity, throughout the remainder of the proof
we fix $\varphi$ and $t$ and write
\[ \Omega:=\Omega_t^\varphi, \qquad u_q:=u_{q,t}^\varphi, \qquad \lambda_q:=\lambda_{q,t}^\varphi, \qquad u'_q[h]:= u_{q,t}^{\prime,\varphi}[h],\qquad \theta=\theta_t,\qquad \beta=\beta_t. \]
Observe that $\mathcal{Q}_q(h,k)=2a_q(u'_q[h],u'_q[k])+2\int_{\partial\Omega}((H-\beta)|\nabla u_q|^2+\frac{q}{2}v_{u_q}|\nabla u_q|-\nabla^2u_q[\nabla u_q]\theta)\, hk \,\,d\mathcal{H}^2$.
Moveover by Lemma \ref{le:a_q coerciva e stime},
\begin{align}
    |a_q(u'_q[h],u'_q[k])-&a_0(u'_0[h],u'_0[k])|\\&\leq|(a_q-a_0)(u'_q[h],u'_q[k])|+|a_0(u'_q[h]-u'_0[h],u'_q[k])|+|a_0(u'_0[h],u'_q[k]-u'_0[k])|
    \\& \leq o_q(1)||h||_{H^{1/2}(\partial \Omega)}||k||_{H^{1/2}(\partial \Omega)} 
\end{align}
and also by Theorem \ref{thm:convergence}
\begin{align}
\Bigg|\int_{\partial\Omega}\Big((H-\beta)|\nabla u_q|^2+\frac{q}{2}v_{u_q}|\nabla u_q|-\nabla^2u_{q}[\nabla u_{q}]\theta_t \Big)\, hk \,\,d\mathcal{H}^2-\int_{\partial\Omega}&\Big((H-\beta)|\nabla u_0|^2-\nabla^2u_{0}[\nabla u_{0}]\theta_t \Big)\, hk \,\,d\mathcal{H}^2\Bigg|\\&\leq o_q(1)||h||_{H^{1/2}(\partial \Omega)}||k||_{H^{1/2}(\partial \Omega)} 
\end{align}
\end{proof}

\begin{theorem}[First and Second Shape Derivatives]\label{thm:shape_derivatives}
Let $q<q^*$ (see Theorem~\ref{thm:convergence}) and $u_{q,t}^\varphi$ be the unique positive solution of $E_q(\Omega_t)$. The first and second shape derivatives of the eigenvalue $\lambda_{q,t}$ with respect to the deformation parameter $t$ are given by:
\begin{equation}\label{eq:first_derivative}
    \lambda_q'(t) = - \int_{\partial \Omega_t} |\nabla u_{q,t}^\varphi|^2 (\phi\nu \cdot \nu_t) \, d\mathcal{H}^2
\end{equation}
and
\begin{align} \label{eq:second_derivative}
    \lambda_q''(t) =& 2 \int_{\Omega_t} |\nabla u_{q,t}^{\prime,\varphi}|^2 \, dx - 2 \lambda_{q.t} \int_{\Omega_t} (u_{q,t}^{\prime,\varphi})^2 \, dx + q \int_{\Omega_t} v_{u_{q,t}^{\prime,\varphi}} u_{q,t}^{\prime,\varphi} \, dx -2\int_{\partial \Omega_t} (\phi\nu\cdot \nu_t)\nabla^2u_{q,t}^\varphi[\nabla u_{q,t}^\varphi]\cdot \phi\nu^\tau\\ & + 2\int_{\partial \Omega_t} \left( H_t |\nabla u_{q,t}^\varphi|^2 + \frac{q}{2} v_{u_{q,t}^\varphi} |\nabla u_{q,t}^\varphi| \right) (\phi\nu \cdot \nu_t)^2 \, d\mathcal{H}^2-\int_{\partial \Omega_t}|\nabla u_{q,t}^\varphi|^2\text{div}(\phi\nu)(\phi\nu\cdot \nu_t),
\end{align}
where $\nu_t$ is the outward unit normal to $\partial \Omega_t$, $H_t = \mathrm{div}_{\Gamma} \nu_t$ is the mean curvature of the boundary, and $u_{q,t}^{\prime,\varphi} =u_{q,t}^{\prime,\varphi} [\phi\nu]$ i.e. the unique solution of:
\begin{equation}\label{eq:shape_derivative_pde}
    \begin{cases}
        -\Delta u_{q,t}^{\prime,\varphi} = \lambda'_q(t) u_{q,t}^\varphi + \lambda_{q,t} u_{q,t}^{\prime,\varphi} - \frac{q}{2} v_{u_{q,t}^{\prime,\varphi}} & \text{in } \Omega_t, \\[4pt]
        u_{q,t}^{\prime,\varphi} = - \frac{\partial u_{q,t}^\varphi}{\partial \nu_t} \big(\frac{x}{|x|} \cdot \nu_t\big)\phi & \text{on } \partial \Omega_t.\\
        \int_{\Omega_t}u_{q,t}^{\prime,\varphi}u_{q,t}^\varphi = 0
    \end{cases}
\end{equation}
\end{theorem}
\begin{proof}
    The formal computations of $\lambda_q'(t)$ and $\lambda_q''(t)$ follow the same step of \cite[Appendix]{BrascoDePhilippisVelichkov}. Differently from \cite{BrascoDePhilippisVelichkov}, we do not work with a volume preserving path and this motivates the presence of the term $\int_{\partial \Omega_t}|\nabla u_{q,t}^\varphi|^2\text{div}(\phi\nu)(\phi\nu\cdot \nu_t)$ which is not $0$ a priori.
    Moreover Theorem \ref{thm: dini} validates the existence of the derivatives.
\end{proof}

\begin{remark}
We give a geometric interpretation of the quadratic form $\mathcal{Q}_{q,t}^\phi(h,h)$ and explain its relation to the normal velocity of the deformation. Fix $\phi$ and $t$, and set $\Omega=\Omega_t^\phi$. For sufficiently regular $h$, consider the deformation
\[
\partial\widehat\Omega_s =\left\{x+s\,h(x)\frac{x}{x\cdot\nu_t(x)}:\ x\in\partial\Omega\right\}.
\]
Then the quadratic form defined above represents the second variation along this deformation:
\[
\mathcal{Q}_{q,t}^\phi(h,h)=\left.\frac{d^2}{ds^2}E_q(\widehat\Omega_s)\right|_{s=0}.
\]

Here, $h$ is the prescribed normal velocity on $\partial\Omega$. Since we restrict ourselves to radial deformations (see \eqref{eq:radial deformation of B1}), this uniquely determines the boundary velocity field:
\[
V_h(x)=h(x)\frac{x}{x\cdot\nu_t(x)},\qquad V_h\cdot\nu_t=h.
\]
Thus, varying $h$ amounts to changing the radial deformation while keeping the reference domain $\Omega_t^\phi$ fixed. In particular, the tangential component of the velocity is also determined by $h$.

This formulation allows us to treat the normal velocity as an independent variable on each fixed domain and to study the dependence of the second variation on $q$ through the associated bilinear forms. We can therefore investigate convergence as $q\to0$ in operator norm, with estimates uniform in $\phi$ and $t$, thereby obtaining control of the second derivative $\lambda_q''(t)$ along the original paths.
\end{remark}

\section{Proof of the main result}
We are now in position to prove Theorem~\ref{thm:intro-rigidity}.

\begin{proof}[Proof of Theorem~\ref{thm:intro-rigidity}]
Let $\phi\in \mathcal{A}_{\delta_1}$ and let $\partial\Omega_t=\{x+tx\phi(x):x\in\partial B_1\}$. 
 Since the optimal function on the ball is radial by Theorem \ref{thm:simmetric solution on ball}, let us define $\mu_q$ to be the constant value of $|\nabla u_{q,B_1}|^2$ on $\partial B_1$. So we consider the function\[t\to j_q(t):=E_q(\Omega_t)+\mu_q\text{Vol}(\Omega_t)+\tilde C|\text{Bar}(\Omega_t)|^2+\tilde C|\text{Vol}(\Omega_t)-\text{Vol}(B_1)|^2,\] where $\tilde C$ will be chosen later, and we perform a Taylor-Lagrange expansion around $0$ and we evaluate it in $t=1$, so we get that
\begin{equation}\label{eq:}
j_q(1)=j_q(0)+j_q'(0)+\frac{1}{2}j_q''(0)+\int_0^1(1-s)(j_q''(s)-j_q''(0))\,ds.
\end{equation}
where, recalling \eqref{eq:first_derivative} and \eqref{eq:second_derivative},
\[
j_q'(0)=\lambda'_q(0)+\mu_q \text{Vol}'(B_1).
\]
In particular, by Theorem \ref{thm:simmetric solution on ball} and the definition of $\mu_q$, 
\begin{equation}
    j_q'(0)=0.
\end{equation}
Since \[\lambda_q''(t)=\langle \mathcal{L}^\phi_{q,t} \phi\nu\cdot \nu_t, \phi\nu\cdot \nu_t\rangle,\]
then by Lemma \ref{lem:convergence operator}, 
\begin{equation}\label{eq:rigidity1}
    \lambda_q''(t)=\langle \mathcal{L}^\phi_{0,t} \phi\nu\cdot \nu_t, \phi\nu\cdot \nu_t\rangle +o_q(1)||\phi\nu\cdot\nu_t||^2_{H^{1/2}(\partial  \Omega_t)}.
\end{equation}
Moreover $\mu_q\text{Vol}''(\Omega_t)[\phi \nu\cdot \nu_t,\phi \nu\cdot \nu_t]=2\mu_q\int_{\partial \Omega_t}\frac{(\phi \nu\cdot \nu_t)^2}{x\cdot\nu_t}d\mathcal H ^2$ and so it is easy to prove that it converge in the operatorial sense, thus it holds 
\begin{equation}\label{eq:rigidity2}
   j_q''(t)=j_0''(t) +o_q(1)||\phi\nu\cdot\nu_t||^2_{H^{1/2}(\partial  \Omega_t)}.
\end{equation}
Since $\langle\mathcal{L}^\phi_{0,t} \phi \nu\cdot \nu_t, \phi\nu\cdot \nu_t\rangle=\lambda''_0(t)$, by \cite[Theorem 1.4, Section 3.3]{DambrineLamboley} and up to choosing $\tilde C$ big enough, we obtain that

\begin{equation}\label{eq:sacre stime}
    j_0''(0)\geq C||\phi||^2_{H^{1/2}(\partial B_1)},\quad  |j''_0(s)-j_0''(0)|\leq \omega(||\phi||_{C^{2,\alpha}(\partial B_1)}||)\phi||^2_{H^{1/2}(\partial B_1)} 
\end{equation}
where $\omega$ is a modulus of continuity.
Thus
 \begin{align}\label{eq:rigidity3}
     j_q(1)\geq& j_q(0)+\frac{1}{2}j_0''(0)-o_q(1)||\phi||^2_{H^{1/2}(\partial B_1)}+\int_0^1(1-s)(j''_0(s)-j_0''(0))\,ds
     \\\geq & j_q(0)+\Big(C-o_q(1)-\omega(||\phi||_{C^{2,\alpha}(\partial B_1)}||)\Big)||\phi||^2_{H^{1/2}(\partial B_1)}
 \end{align}

Observe that the constant $C$ is independent on $q$ and $\delta$. Then by Theorem \ref{thm:main}, there exists $0< q_2\leq q_1$ for which we can reduce to show the rigidity on the ball between $C^{2,\alpha}$-nearly spherical with $C^{2,\alpha}$ norm lower than $\delta_1$ and with $\text{Bar}(\Omega)=0$. For such set, 
\[ j_q(1)- j_q(0)=E_q(\Omega)-E_q(B_1) \]
and so we can apply \eqref{eq:rigidity3} to obtain that the ball is the only minimizer, up to translation and sets of zero Lebesgue measure. By standard capacitary arguments, the ball is the only minimizer up to translation and sets of zero capacity.
\end{proof}
\begin{remark}\label{rem:spiegazione convergenza operatoriale}
We point out that to obtain \eqref{eq:rigidity1} it is needed the convergence of the associated operator proved in Lemma \ref{lem:convergence operator}. Indeed the pointwise convergence (or even uniform in $t\in[0,1]$) $\lambda_q''(t)\to \lambda_0''(t)$ is not enough since it leads to the estimate
\[ \lambda''_q(t)=\lambda''_0(t)+o_q(1),\]
and so  
\begin{align}
     j_q(1)\geq & j_q(0)+C||\phi||^2_{H^{1/2}(\partial B_1)}+o_q(1).
 \end{align}
 So we would only get the minimality of the ball among $C^{2,\alpha}$-nearly spherical domain such that $||\phi||^2_{H^{1/2}(\partial B_1)}\geq o_q(1)$ but this is not enough to conclude.
\end{remark}
\begin{remark}
    We want to stress the importance of the study of the functional $t\to j_q(t)$ instead of $t\to E_q(\Omega_t)$. Indeed, since the functional $E_q(\Omega_t)$ is translation invariant and minimized on a volume constraint, it would be false the inequality \[ E_q(\Omega^\phi)\geq E_q(B)+C||\phi||^2_{H^{1/2}(\partial B_1)},\quad \forall \phi\in\mathcal{A}_{\delta_1}\]
    For example, if we translate the ball with a small vector, then new set is a nearly spherical domain with same energy. This correction allows us to work with a sequence of sets $(\Omega_t)_t$ that are neither subject to volume constraints nor centered at the origin. This idea comes from \cite{DambrineLamboley}.
\end{remark}
\begin{remark}
    The reader could also be interested to apply these strategy to other kind of perturbations, different from the one following a normal path $\partial\Omega_t=\{x+tx\phi(x):x\in\partial B_1\}$. For example, these computations still holds for the volume preserving path introduced in \cite[Appendix]{BrascoDePhilippisVelichkov}. In our case this could not lead to the conclusion because there is no proof, up to our knowledge, of the inequality in \eqref{eq:sacre stime} which such path. Despite it, we still believe it should holds.
\end{remark}

\section*{Note on the use of AI}
The idea of the strategy of this paper is due only to the authors (who were inspired by the previous seminal works~\cite{BrascoDePhilippisVelichkov} and~\cite{DambrineLamboley}). The proof of Lemma~\ref{le:a_q coerciva e stime} was obtained through several interactions with AI models.
Moreover AI models have been used for literature review and for checking typos.

\vspace{0.2cm}

\paragraph{\bf Acknowledgements} 
The authors have been partially supported by the MUR and the European Union via PRIN2022 project P2022R537CS.
The authors are members of
INdAM-GNAMPA and have been partially supported by the INdAM-GNAMPA projects CUP E5324001950001 and CUP E53C25002010001.  \\
R. Moraschi have been partially supported by the LYSM lab of CNRS and INdAM.

\bibliographystyle{abbrv}
\bibliography{bibliografia}

\end{document}